\documentclass[12pt]{amsart}
\usepackage{fullpage}
\usepackage{color}
\usepackage{bookman}
\usepackage{xcolor}

\usepackage{hyperref}
\input xy
\xyoption{all}

\newcommand{\Gal}{\mbox{Gal}}

\newcommand{\id}{\mbox{id}}

\usepackage{color, graphicx}
\usepackage{amsmath}
\usepackage{amssymb}

\usepackage{calrsfs}

\newcommand{\wesley}[1]{\bigskip \textcolor{blue}{\textbf{$\clubsuit \clubsuit \clubsuit$ Wesley:} #1} \bigskip}

\theoremstyle{plain}

\newtheorem{theorem}{Theorem}[section]
\newtheorem{lemma}[theorem]{\bf Lemma}

\newtheorem{corollary}[theorem]{\bf Corollary}
\newtheorem{proposition}[theorem]{\bf Proposition}

\theoremstyle{definition}
\newtheorem{definition}[theorem]{\bf Definition}
\newtheorem{question}[theorem]{\bf Question}
\newtheorem{conjecture}[theorem]{\bf Conjecture}
\newtheorem{example}[theorem]{\bf Example}

\theoremstyle{remark}
\newtheorem{remark}[theorem]{\bf Remark}

\newtheorem{comment}{\bf Comment}

\newcommand{\calB}{{\mathcal B}}

\newcommand{\N}{{\mathbb N}}
\newcommand{\Q}{{\mathbb Q}}
\newcommand{\R}{{\mathbb R}}
\newcommand{\Z}{{\mathbb Z}}

\newcommand{\be}{\begin{enumerate}}
\newcommand{\ee}{\end{enumerate}}

\newcommand{\rat}{\Q}

\newcommand\cyr{%
\renewcommand\rmdefault{wncyr}%
\renewcommand\sfdefault{wncyss}%
\renewcommand\encodingdefault{OT2}%
\normalfont 
\selectfont}
\DeclareTextFontCommand{\textcyr}{\cyr}

\title{Computability in infinite Galois theory and algorithmically random algebraic fields}

\author{Wesley Calvert}
\address{School of Mathematical and Statistical Sciences, Mail Code 4408\\ 1245 Lincoln Drive\\ Southern Illinois University\\ Carbondale, Illinois 62901}
\email{wcalvert@siu.edu}
\author{Valentina Harizanov}
\address{Department of Mathematics\\ The George Washington University\\ Washington, DC 20052}
\email{harizanv@gwu.edu}
\author{Alexandra Shlapentokh}
\address{Department of mathematics, East Carolina University, Greenville, NC 27858}
\email{shlapentokha@ecu.edu}


\thanks{This material is partially based upon work supported by the National Science Foundation under grant DMS-1928930 while all three authors participated in a program hosted by the Mathematical Sciences Research Institute in Berkeley, California during Fall 2020 and Summer 2022.  The first author was partially supported by a Fulbright-Nehru Senior Research Fellowship.  The second author was partially supported by NSF FRG grant DMS-2152095.  The third author was partially supported by NSF FRG grant DMS-2152098.  The authors are grateful to Philip Dittmann for helpful comments, and to the anonymous referees for a very careful reading.
}

\begin{document}
\maketitle

\begin{abstract} We introduce a notion of algorithmic randomness for algebraic fields.  We prove the existence of a continuum of algebraic extensions of $\rat$ that are random according to our definition.  We show that there are noncomputable algebraic fields which are not random. We also partially characterize the index set, relative to an oracle, of the set of random algebraic fields computable relative to that oracle.

In order to carry out this investigation of randomness for fields, we develop computability in the context of infinite Galois theory (where the relevant Galois groups are uncountable), including definitions of computable and computably enumerable Galois groups and computability of Haar measure on the Galois groups.
\end{abstract}

\section{Introduction}
We can often understand a class of structures better by understanding the ``typical" elements of the class.  In this paper, we propose a definition of a typical algebraic field.  Intuitively, we would expect that a typical object is the most likely result of a selection at random from the class, perhaps with respect to some probability distribution.  Thus, we refer to our typical fields as "random fields."

The study of random elements has been carried out for many classes of objects.  Perhaps the most well-developed and fruitful example of such a study is the case of random graphs \cite{ErdosRenyi1960,ErdosRenyiRG1,BollobasRG}.  While the randomness discussed in that work was not algorithmic, it demonstrates the potential benefits of the investigation of random elements.  In a benchmark example of the effectiveness of this technique, Pinsker used random graphs to answer a major question about the existence of a class of graphs (where the original problem made no mention of randomness) \cite{Pinsker1973}.  This leads us to hope that understanding random fields might lead to new insights about all fields.

There has also been work on random groups (including the work described in \cite{Gromov,CordesEtAl2018} and \cite{FranklinHoKnight2022}, and random structures more generally \cite{Fouche2013,GlasnerWeiss2002,Khoussainov2016,Khoussainov2014,HTKhoussainovTuretsky2019}.  These papers, collectively, give many definitions of typical elements of many classes of structures. It appears, however, that it is more difficult to make a meaningful definition of a typical element when discussing classes of structures that exhibit greater complexity.  

It turns out that there is a combinatorial property (trivial definable closure) that determines whether a class of structures will have a single isomorphism type representing a typical element (in the sense of a suitably invariant probability measure on the class concentrating on that isomorphism type) \cite{AFP2016,AFKP2017,AFP2019}.  Graphs have this property, but groups and, critically for this paper, fields, do not.


In view of this, the problem of describing a typical, or random, field, is more complicated than the analogous problem for graphs.  In this paper we will give a definition that we argue captures the notion of typicality on algebraic fields.

To achieve our results on random fields, we develop computability theory for absolute Galois groups, based on Krull topology.  Galois groups of many infinite fields are uncountable and therefore the usual notions of computability are not applicable in this context. However, the graphs of absolute Galois groups of countable fields (defined in Section \ref{SecCompGalois}) are countable and therefore can be used to define a computable absolute Galois group.  The fixed field of a Galois group which is computable in this sense turns out to be computable in the usual way confirming that our definition is a reasonable one. Things become more complicated when it is necessary to define c.e. absolute Galois groups that are not computable.  The reader can find this discussion in Section \ref{SecCompGalois}.

To make use of our definition of a computable absolute Galois group we had to make sure that the computability of a group and its fixed field extends to various properties of the field in an expected manner. The technical results pertaining to these matters can be found in Sections \ref{backgdcomp} and \ref{HaarSec} and especially in Section \ref{SecCompGalois}.

For reasons above, this paper combines results from computability and computable structure theory, algorithmic randomness, and infinite Galois theory. As we point out below, we included some basic facts from these areas to make the paper readable by a wider audience. 

In this paper we state our results in the language of algebraic extensions of $\Q$ to simplify the presentation.  Since $\Q$ has a splitting algorithm, we can construct a computable copy $\overline{\Q}$ of an algebraic closure in which the domain of $\Q$ is a computable subset.  However, our work actually shows a good deal more.  Everything in the present paper applies to any computable field with a splitting algorithm, e.g.,\ a countable function field.

\subsection{The structure of the paper}
Because we hope to accommodate readers from diverse backgrounds, we have attempted to make the paper self-contained.  In Section \ref{backgdcomp}, we review standard background in computability. Section \ref{FixedACl} introduces an important assumption used throughout the paper.  Section \ref{HaarSec} gives a similar background on Haar measure and absolute Galois groups, with Section \ref{compmeas} describing computability in this context.

Section \ref{SecCompGalois} begins the technical heart of the paper, establishing important relationships between the computable structure theory of algebraic fields and that of their absolute Galois groups.  This section establishes the vocabulary that allows Section \ref{SecRand} to finally define random fields, prove their existence, and make some preliminary investigations into their characterization and properties.  Section \ref{SecIdx} addresses the index set problem for random fields.

\section{Background on Computability}\label{backgdcomp}
\subsection{Computable and c.e.\ sets and computable functions}\label{SecComputability}
We follow the standard notations and definitions of computability theory.   Good general references for this section are \cite{soare,weber,ashknight}.
\subsubsection{Computable and computably enumerable sets}
A set $C$ of natural numbers (or of tuples of natural numbers) is \emph{computable} if there is a decision procedure for identifying its elements. That is, there is a Turing machine that on an input $x$ always halts and outputs $1$ if $x \in C$ and outputs $0$ if $x \notin C.$ Computable sets encode decidable problems (that is, problems that can be solved algorithmically). All finite sets are computable.  Appealing to Church's Thesis, we say that a function is computable if and only if there is an algorithm to compute it. 

A set $E$ of natural numbers is  \emph{computably enumerable} (abbreviated by c.e.) if $E$ is empty or there is a computable unary function $f$ such that 
\[
E =\text{ran}(f) =\{f(0) ,f(1) ,\ldots \}.
\]
Hence every computable set is c.e.. For an infinite c.e.\ set an enumeration $f(0),f(1),\ldots $ can be modified by eliminating repetitions to be one-to-one.  A classical result in computability shows there are c.e.\ sets that are not computable.

A partial function $\psi$ is \emph{partial computable} if there is a Turing machine that on every input in the domain of $\psi$ halts and outputs its value, while on every input that is not in the domain of $\psi$ it computes forever. It can be shown that a set $E$ is c.e.\ if and only if it is the domain of a partial computable function $\psi, $ i.e.,\ $E =\text{dom}(\psi ).$ Moreover, a set $E$ is c.e.\ if and only there is a computable binary relation $R$ such that for every $x$ it is the case that
\begin{center}
$x \in E \Leftrightarrow ( \exists y)R (x ,y)\text{.}$
\end{center}\par

All c.e.\ sets can be simultaneously algorithmically enumerated by effectively enumerating all Turing machines. In other words, there is is a computable enumeration of all unary partial computable functions and their domains.

Clearly, the complement of a computable set is computable. On the other hand, there are c.e.\ sets --- exactly the noncomputable c.e.\ sets --- with complements that are not c.e..  A well-studied example of a noncomputable c.e.\ set is the diagonal halting set $K$.  The set $K$ consists of all inputs $e$ on which the Turing machine with index $e$ (or the $e$-th computable function $\varphi_e$) halts. That is,
\[
K =\{e :e \in W_{e}\}=\{e: \varphi_e(e) \mbox{ halts}\},
\]
where $W_e$ is the c.e.\ set that is the domain of the function $\varphi_e$ computed by the Turing machine with index $e$.

\subsubsection{Arithmetical hierarchy} We have the following classification of subsets of natural numbers and, more generally, of relations on natural numbers.  $\Pi _{n}^{0}$ and $\Sigma _{n}^{0}$ sets (relations) are levels in the \emph{arithmetical hierarchy} obtained from computable relations by applying existential and universal quantifiers. More precisely, a set $A$ is $\Sigma _{0}^{0} =\Pi _{0}^{0}$ if it is computable. For $n >0 ,$ a set $A$ is $\Sigma _{n}^{0}$ if there is a computable $(n +1)$-ary relation $R$ such that for every $a \in \mathbb{N}$, \medskip 

\begin{center}
$a \in A \Leftrightarrow ( \exists x_{0})( \forall x_{1}) \cdots (Q x_{n -1}) R (a ,x_{0} ,x_{1} ,\ldots  ,x_{n -1})\text{}\text{}$, \medskip 
\end{center}\par
\noindent where $Q$ is $ \exists $ if $n$ is an odd number, and $Q$ is $ \forall $ if $n$ is an even number. 

$\Pi _{n}^{0}$ sets are defined similarly starting with the universal quantifier.  Clearly, the complement of a $\Sigma _{n}^{0}$ set is a $\Pi _{n}^{0}$ set and vice versa.  We say that a set is $\Delta _{n}^{0}$ if it is both $\Sigma _{n}^{0}$ and $\Pi _{n}^{0}$. It follows that $\Sigma _{1}^{0}$ sets are the c.e.\ sets, $\Pi _{1}^{0}$ sets are the co-c.e.\ sets, and  $\Delta _{0}^{0} =\Delta _{1}^{0}$ sets are the computable sets. A set is called \emph{arithmetical} if it is $\Pi _{m}^{0}$ for some $m$ (or $\Sigma _{m}^{0}$ for some $m$).  

The superscript $0$ in the notation $\Pi^0_n$ and $\Sigma^0_n$ indicates that all quantifiers will range only over natural numbers (and elements of structures that have been indexed by natural numbers).  A superscript of $1$ or more would indicate quantification over functions or higher-type objects, and will play no role in the present paper.

Given a set complexity class $\mathfrak{C}$, such as $\Sigma _{n}^{0}$ or $\Pi _{n}^{0}$, we say that a set $X$ of natural numbers is $m$-\emph{complete} $\mathfrak{C}$ if $X$ is in $\mathfrak{C}$, and there is a computable reduction of  every set $Y$ in $\mathfrak{C}$ to $X$ (i.e., $X$ is $\mathfrak{C}$-hard). This reduction is a computable function $f :\mathbb{N} \rightarrow \mathbb{N}$ such that for every $n \in \mathbb{N}$, we have:\medskip

\begin{center}$n \in Y \Leftrightarrow f(n) \in X$.\end{center} 

\noindent Since the function $f$ can be many-one function we call this completeness $m$-completeness. For example, the halting set is $\Sigma _{1}^{0}$ $m$-complete.

\subsubsection{Turing reducibility}

For $Y \subseteq \mathbb{N}$, let\medskip
\begin{center}

$\varphi _{0}^{Y} ,\varphi _{1}^{Y} ,\varphi _{2}^{Y} ,\ldots $ \end{center} \medskip

\noindent be a fixed effective enumeration of all unary $Y$-computable functions --- that is, functions which are computable using $Y$ as an oracle. If an oracle $Y$ is computable, then it is not needed, so we omit the superscript $Y$. For sets $X$ and $Y$, we write $X  \leq _{T}Y$ if $X$ is Turing reducible to $Y$ --- that is, if the characteristic function of $X$ is given by $\varphi _{e}^{Y}$ for some $e$.  The partial order relation $ \leq _{T}$ gives rise to an equivalence relation, the equivalence classes of which are called Turing degrees. The Turing degree of $\varnothing$ or of any computable set is denoted $\mathbf{0}$.




For a set $Y,$ the jump of $Y$ is defined generalizing $K$ as follows:

\begin{center}$Y^{ \prime } =\{e : \varphi _{e}^{Y}(e) \mbox{ halts}\}.$\end{center}

\noindent Hence $\varnothing^{ \prime } =K.$ It can be shown that $Y <_{T}Y^{ \prime }$. The jump operator can be iterated: $Y^{(k +1)} =(Y^{(k)})^{ \prime }$, where $Y^{0} =Y.$

 For $n \geq 1$, let $\mathbf{y}^{(n)} =\deg  (Y^{(n)})$. It can be shown that a set is $\Sigma _{n}^{0}$ if it is computably enumerable in (or relative to) $\mathbf{0}^{(n -1)}$. Hence a set $X$ is \emph{arithmetical} if $X \leq_{T} \varnothing ^{(k)}$ for some $k \geq 0$.


\subsubsection{Enumeration reducibility}
A set $X$ is \emph{enumeration reducible} to a set $Y$, denoted $X \leq _{e}Y$, if we can computably enumerate the elements of $X$ from an enumeration of the elements of $Y$, where the enumeration of $X$ does not depend on the order is which $Y$ is enumerated. That is, $X =\Psi ^{Y}$, where $\Psi $ is some enumeration operator. If $X \leq _{e}Y$ then $X$ is computably enumerable in $Y$. Moreover, Selman showed that $X \leq _{e}Y$ if and only if for every set  $C$, if $Y$ is computably enumerable in $C$, then $X$ is computably enumerable in $C$.  The partial order relation $ \leq _{e}$ of sets gives rise to an equivalence relation, the equivalence classes of which are called enumeration degrees. There are also enumeration degrees of partial functions, called partial degrees. 

\subsubsection{Computable formulas}
We will only consider countable structures for computable languages. The universe $A$ of an infinite countable structure $\mathcal{A}$ can be identified with the set of natural numbers. If $L$ is the language of $\mathcal{A}$, then $L_{A}$ is the language $L$ expanded by adding a constant symbol for every $a \in A$, and $\mathcal{A}_{A} =(\mathcal{A} ,a)_{a \in A}$ is the corresponding expansion of $\mathcal{A}$ to $L_{A}$. The \emph{open}\emph{\ diagram}
of a structure $\mathcal{A}$, $D (\mathcal{A})$, is the set of all quantifier-free sentences of $L_{A}$ true in $\mathcal{A}_{A}$. A structure is \emph{computable} if its open diagram is computable. A structure for a finite language is computable if its domain is a computable set and its functions and relations are computable.\\

Computable $\Sigma _{0}$ and $\Pi _{0}$ formulas are just the finitary quantifier-free formulas (that is, the quantifier-free formulas involving only finitely many disjuctions, conjunctions and quantifiers).  Let $n >0$. A \emph{computable} $\Sigma _{n}$ \emph{formula} is a c.e.\ disjunction of formulas

\begin{center}$ \exists \overline{u}\, \psi  (\overline{x} ,\overline{u})$,\end{center}\par
\noindent where $\psi $ is a computable $\Pi _{m}$ formula for some $m <n$.\bigskip

A \emph{computable} $\Pi _{n}$ \emph{formula} is a c.e.\ conjunction of formulas

\begin{center}$ \forall \overline{v}\, \theta  (\overline{y} ,\overline{v})\text{}$,\end{center}\par
\noindent where $\theta $ is a computable $\Sigma _{m}$ formula for some $m <n$.

In a computable structure, a computable $\Sigma _{n}$ formula defines a $\Sigma _{n}^{0}$ set, and a computable $\Pi _{n}$ formula defines a $\Pi _{n}^{0}$ set. 
For more on computable structures and computable formulas see \cite{ashknight,fokharmel}.

\subsubsection{Immune sets}
In an attempt to construct a  c.e.\ set of Turing degree strictly between the computable one and the degree of the halting set, Post introduced sets with "thin" complements with respect to c.e.\ sets.

\begin{definition} A set of natural numbers is \emph{immune} if it is infinite and does not contain any infinite c.e.\ subset.\end{definition}

The complements of immune sets may or may not be c.e. Those that are c.e.\ are called simple sets and were first constructed by Post. There is further proper strengthening of immune sets into hyperimmune, hyperhyperimmune, and cohesive sets.  While there are countably many simple sets, it can be shown that there are continuum many cohesive sets.  

\begin{definition} A set $C \subseteq \mathbb{N}$ is \emph{cohesive} if $C$ is infinite and for every c.e.\ set $W\text{,}$ either $W \cap C$ or $\overline{W} \cap C$ is finite.

(Here, $\overline{W}$ is the complement of $W$. Hence a cohesive set $C$ is indecomposable into two infinite parts by a c.e.\ set $W$.)\end{definition}  

 

\begin{definition} A set $C \subseteq \mathbb{N}$  is \emph{$r$-cohesive} if $C$ is infinite and for every computable\ set $W\text{,}$ either $W \cap C$ is finite or $\overline{W} \cap C$ is finite.
\end{definition}

Every cohesive set is $r$-cohesive, but the converse is not true. Every  $r$-cohesive set is immune, but the converse is not true.

\begin{lemma} 
\label{le:cohesive} An $r$-cohesive set cannot have immune complement. Hence a cohesive set cannot have immune complement.\end{lemma}
\begin{proof} Assume that $C$ is $r$-cohesive. Fix an infinite co-infinite computable set $R$. Then  either $R \cap C$ is finite or $\overline{R} \cap C$ is finite. If
$R \cap C$ is finite, then $R -C$ is an infinite c.e.\ (in fact, computable) set such that $R -C \subseteq \overline{C}$. If $\overline{R} \cap C$ is finite, then $\overline{R} -C$ is an infinite c.e.\ (in fact, computable) set such that $\overline{R} -C \subseteq \overline{C}$. Hence $\overline{C}$ is not immune.\end{proof}

We now describe a property of sets of natural numbers that we will use in Section \ref{SecRand}.

\bigskip

(*)\ $X \subseteq \mathbb{N}$
 is an infinite co-infinite set such no co-infinite superset $S \supseteq X$ has an infinite c.e.\ subset (i.e., $S$ is immune). \bigskip

\begin{lemma}
\label{star}
There is no set $X$ with property (*).
\end{lemma} 

\begin{proof} If $X$ has property (*), then $X$ must be immune. Otherwise, $X$ has an infinite c.e.\ subset, so $S =X$ has an infinite c.e.\ subset, contradicting property (*).

If $X$ has property (*), then $\overline{X}$ must be cohesive. Otherwise, there is a c.e. set $W$ such that both $W \cap \overline{X}$  and $\overline{W} \cap \overline{X}$ are infinite. Let $S =X \cup W .$ Then $S$ is co-infinite since $\overline{S} =\overline{W} \cap \overline{X}$  is infinite, and $S$ contains an infinite c.e.\ subset $W ,$ contradicting property (*).

By Lemma \ref{le:cohesive}, it is not possible to have a cohesive set with immune complement.\end{proof}

\subsection{Computable structures}
To study computability on a countable infinite structure, we construct a bijection from elements of the structure onto the natural numbers.  The operations on the structure are then translated into the maps over natural numbers.  It is tempting to say that a structure is computable (or c.e.) if and only if the image in the natural numbers is computable (resp.\ c.e.) and its operations are all computable.  

Unfortunately, this definition does not really differentiate between computable and c.e.\ structures because we can always construct a computable isomorphic copy of a computably enumerable field.  To understand the difference between computable and c.e.\ structures, one has to look at a problem of simultaneously representing two structures computably within a structure containing them both.  

The issue of simultaneous computable enumeration is easy to see when it comes to fields.  By a famous theorem of Rabin \cite{rabin} we know that any computable (in the sense above) field has a computable algebraic closure (in the sense above).  However, the theorem does not guarantee a simultaneous computable presentation of the algebraic closure and the original field within it.  

The same result of Rabin also tells us that the simultaneous computable presentation of the field and its algebraic closure is possible if and only if the original field had a splitting algorithm (see Definition \ref{SplittingAlg}).  In other words, given a computable field in the sense above without a splitting algorithm, we have a choice for a construction of the algebraic closure: either we make the original field computable in the sense above and we have a c.e.\ algebraic closure, or we have a c.e.\ presentation of the original field and a computable algebraic closure.



\begin{definition}[Computable fields and c.e.\ fields] Let $F$ be a computable field with a splitting algorithm, and $\overline{F}$ a computable algebraic closure of $F$ such that $F$ is a computable subset of $\overline{F}$.  Then a subfield of $\overline{F}$ is said to be \emph{computable} if its set of elements is a computable subset of the set of natural numbers.  A subfield of $\overline{F}$  is said to be \emph{computably enumerable} if and only if its set of elements is computably enumerable.


\end{definition}

\begin{definition}\label{UnifComp} We say that a sequence of fields $(F_i : i \in \mathbb{N})$ is \emph{uniformly computable} if and only if there is a computable function $\phi$ such that $\phi(i)$ is the index for a Turing machine computing the characteristic function of $F_i$.\end{definition}

\begin{remark} There exist sequences $(F_i : i \in \mathbb{N})$ of computable fields which are not uniformly computable.   For instance, let $(p_i : i \in \mathbb{N})$ be an enumeration of the distinct rational primes.  Then the sequence of fields $\left(\Q(\sqrt{p_i}) : i \in \emptyset'\right)$ is not uniformly computable, since $\emptyset'$ is not computable.  However, each individual field in the sequence is computable. \end{remark}




\subsection{A fixed computable algebraic closure of $\Q$}\label{FixedACl} 
For the rest of the paper we fix a computable bijection $\sigma: \overline{\Q} \longrightarrow \omega$ such that the images of the graphs of addition and multiplication are also computable.  As noted above $\sigma(\Q)$ will be a computable set.


Notice that even in this environment, many properties may remain ineffective.  For instance, it is not clear that we could effectively determine, given the characteristic function of a field as a subset of $\overline{\Q}$, whether that field is a finite extension of $\Q$, or, if finite, what its degree would be.



\begin{definition}[Splitting Algorithm]\label{SplittingAlg}
Let $K$ be a computable field.  We say that $K$ has a splitting algorithm if there is an effective procedure to determine whether a polynomial with coefficients in $K$ is irreducible over $K$.
\end{definition}

The following result is a part of Rabin's theorem we discussed above.

\begin{proposition}
\label{prop:computable}
Every algebraic extension of $\Q$ within a fixed computable algebraic closure of $\Q$ is computable if and only if it has a splitting algorithm.
\end{proposition}
\begin{proof}
Let $K \subset \overline{\Q}$ be a computable field.  Given a polynomial over $K$ we can find all of its roots in $\overline{\Q}$ and then determine which symmetric functions of the roots lie in $K$.

Conversely, suppose that a field $K$ has a splitting algorithm.  Given an element $x$ of the algebraic closure we find some polynomial $p(T)$ over $\Q$ such that $p(x)=0$ and determine the factorization of $p(T)$ over $K$.  Then $x \in K$ if and only if $p(T)$ has a factor $(T-x)$ over $K$.
\end{proof}

The following results are standard, first proved in  \cite{FrolichShepherdson}.  We adjust them slightly for our context, somewhat more narrow than \cite{FrolichShepherdson}, where work inside a fixed algebraic closure was not assumed.

\begin{lemma}\label{FindSingleGen}
Let $M/K$ be a finite extension of computable fields, given by computable characteristic functions of their domains and the degree of the extension.  Then there exists an effective procedure to find an element $\alpha$ such that $M=K(\alpha)$.
\end{lemma}
\begin{proof}
Let $n=[M:K]$.  Since $K$ is computable, it has a splitting algorithm by Proposition \ref{prop:computable}. We test elements of $M$ until we find an element $\alpha \in M$ satisfying an irreducible polynomial of degree $n$ over $K$.
\end{proof}

\begin{lemma}\label{FindConjugates} There is an effective procedure which will, given a computable field $K$ and an element $x \in \overline{\Q}$, determine the set of conjugates of $x$ over $K$. 
\end{lemma}

\begin{proof}
Since $K$ is computable, we proceed as follows. We find a polynomial $P(t)$ over $\Q$ satisfied by $x$.  We then factor $P(t)$ over $\Q$ to determine an irreducible factor $Q(t)$ of $P(t)$ satisfied by $x$.  We can then find all the other roots $x=x_1,\ldots, x_m$ of $Q(t)$ in the computable algebraic closure, and by considering all possible symmetric functions of the conjugates determine the minimal polynomial $S(t)$ of $x$ over $K$.
\end{proof}

\begin{lemma}
\label{le:compext}
Let $K\subset \overline{\Q}$ be computable.  Then any finite extension of $K$ is also computable, uniformly in the generators of the extension and in $K$.  (Therefore the splitting algorithm for the extension can also be constructed in a uniform fashion from the characteristic function of $K$ and a generator of the extension.) 
\end{lemma}
\begin{proof}
Let $K$ be a computable extension of $\Q$ generated by an element $\alpha \in \Q$.  Since $K$ has a splitting algorithm we can (by Lemma \ref{FindConjugates}) determine all the distinct conjugates $\alpha_1=\alpha, \ldots, \alpha_n$ of $\alpha$ over $K$.  Let $\beta \in \overline{\Q}$.  Applying Lemma \ref{FindConjugates} again, let $\beta_1=\beta, \ldots, \beta_r$ be all the distinct conjugates of $\beta$ over $K$.  If $r>n$, then $\beta \not \in K(\alpha)$.

Suppose now that $r\leq n$ and consider the following  system in the unknowns $a_0,\ldots,a_{n-1} \in \overline{\Q}$:
\[
\sum_{i=0}^{n-1}a_i\alpha_k^i=\beta_{j_k}, k=1,\ldots, n,
\]
where $\beta_{j_1}=\beta$ and $j_k \in \{1,\ldots,r\}$.  The determinant of the system is a Vandermonde determinant.  Thus, for every choice for $j_2,\ldots,j_n$ we can solve the system over $\overline{\Q}$ and see if the solutions are in $K$.  If $\beta \in K(\alpha)$, the system has solutions in $K$ for some choice of $j_2,\ldots,j_{n}$.  Conversely, if for some choice of $j_2,\ldots,j_{n}$ the system has solutions $a_0, \ldots, a_{n-1} \in K$, then $\beta=a_0+a_1\alpha +\ldots +a_{n-1}\alpha^{n-1} \in  K(\alpha)$.
\end{proof}
\begin{corollary}
\label{cor:sequence}
 Let $K$ be a computable field and let $\{\alpha_i\}$  be a computable sequence of elements of $\overline{K}$.  Then there is an effective procedure taking as its input the index of the element of the sequence and generating a characteristic function of $K(\alpha_1,\ldots,\alpha_i)$ and the splitting algorithm for this field.
\end{corollary}

 \subsection{Computability on spaces of functions}\label{SecBaire}

So far we have discussed computability over countable structures containing objects describable by a finite input.  However, to study randomness over algebraic extensions of $\Q$ one has to use some notion of computability for collections of objects that one cannot describe completely using a finite amount of information, e.g., the elements of the absolute Galois group of $\Q$.  

\begin{definition}[Absolute Galois Group]
Let $K$ be an algebraic extension of $\Q$.  Then the absolute Galois group of $K$, denoted $\Gal(\overline{\Q}/K)$, is the group of all automorphisms of $\overline{\Q}$ which restrict to the identity function on $K$.
\end{definition}

These elements, of course, are maps from a countable set to a countable set. In the present section, we describe some of the standard tools for studying objects of this kind from the point of view of computability theory.

 Recall that we are working in a particular computable copy of $\overline{\Q}$, constructed via a fixed computable bijection  $\sigma: \overline{\Q} \longrightarrow \omega$.   Thus, we could identify any map $f: \overline{\Q} \longrightarrow \overline{\Q}$ with a countable sequence $\{a_i\}$ of natural numbers, where $f(\sigma^{-1}(i))=\sigma^{-1}(a_i)$.  Observe that this identification between the set of all maps from $\overline{\Q}$ to $\overline{\Q}$ and the set of all sequences of natural numbers is a bijection induced by $\sigma$.
 
 In view of this identification of functions and sequences, to study computability on the set of all functions from $\overline{\Q}$ to $\overline{\Q}$, it is sufficient to study computability on the space of all infinite sequences of natural numbers.

 The first question of such a study is to define a computable set of sequences.  To make such a definition reasonable one would want an algorithm that determines whether a particular sequence is in a given set.  Unfortunately, an algorithm can take inputs of finite size only and therefore we can never specify the whole sequence as an input.  This leaves us with only one option: we specify a finite part of a sequence.  

 The two obvious options are: we specify initial segments of a given sequence or any finite part of the graph of the corresponding function.  Thus we can possibly ask questions of the following sort about a collection $S$ of sequences:
 \begin{enumerate}
\item Given a finite initial segment of a sequence: $a_1, \ldots, a_m$, is there a sequence in $S$ with such an initial segment?
\item Given two finite collections of integers $m_1, \ldots, m_r$ and $a_1, \ldots, a_r$, we can ask whether $S$ contains a sequence $\{b_i\}$ such that $b_{m_j}=a_j, j=1,\ldots,r$.
\end{enumerate}
Are these two approaches equivalent?  The answer to this question depends on the size of the potential set of values for a position in the sequence.  

If the potential set of entries for any position in the sequence is infinite, then the graph input provides more information.  Indeed, suppose we want to know whether $S$ contains a sequence $\{b_i\}$ with $b_m=a$ for some fixed $m$ and $a$.  To answer this question using initial segments we need potentially to ask a question about every possible initial segment of size $m$ with the last element of the segment equal to $a$.  If the number of such initial segments is infinite then this process might not converge.

Conversely, assuming we can answer questions about finite subsets of the graph of the function corresponding to a given sequence, we can effectively determine whether $S$ contains a sequence with a prescribed initial segment.

At the same time if the number of possible entries for a position in a sequence is finite, then the information provided by finite subsets of the graph and finite initial segments is the same because for any $m$ there will be only finitely many initial segments of size $m$.

Finally, we can ask whether it makes a difference whether we are allowed to specify only one element of the graph or finitely many when asking a question about sequences in $S$.  The answer again depends on the number of possible values for a position in a sequence since asking about a finite part of the graph requires information about potentially infinitely many initial segments.  

Indeed, suppose we want to know whether our set $S$ of sequences contains a sequence with $b_i=a_1$ and $b_j=a_2$.  Assuming $j>i$, to answer the question we need to know if there exists an initial segment of length $j$ containing the above described entries in the positions $i$ and $j$.  As before, in the case of infinitely many values allowed for each position in the sequence, the information about pairs in the graphs of functions in $S$ is not sufficient to answer this question.

At the same time if the number of potential values for each position is finite, then all three approaches are the same.  

There are many implementations of the three approaches above.  For example the initial segment approach has been realized via a function 
\[
f_S(d,r):=\left\{\begin{array}{ll} 1 & \mbox{if $B_r(d) \cap ({^\sigma S}) \neq
  \emptyset$}\\ 0 & \mbox{if $B_{2r}(d) \cap ({^\sigma S)} = \emptyset$}\\ \mbox{$0$
  or $1$} & \mbox{otherwise}\\ \end{array}\right.\]
where $d \in \omega^{<\omega}$, where $r$ is a rational number, and where
$B_r(d)$ is the ball of radius $r$ about $d$ in Baire space
(see \cite{weihrauch}).  Indeed, this approach would give us an additional equivalent condition in Theorem \ref{EqDnf:Computable}.

In the case of an absolute Galois group of a countable field the number of values assigned to each position is always finite because it is bounded by the number of roots of the minimal polynomial of the element in question  over a primary field.

It must be remembered that in identifying functions $\overline{\Q} \to \overline{\Q}$ with $\omega^\omega$, this representation and everything around it --- including the computability, depend on the specific bijection $\mathbb{N} \to \overline{\Q}$, a theme to which we will return in the next section.

\section{Haar Measure}\label{HaarSec}
In this section we will discuss Haar measure on the absolute Galois group of $\Q$. For the most part we follow the presentation in Fried and Jarden (see \cite{friedjarden}).
\subsection{Inverse Limits and Profinite Topology}
In this section we review the notion of inverse limits and the topology arising from them.

Let $I$ be a partially ordered set.  An inverse system $(S_i,\pi_{i,j} : i,j \in I)$ consists of a family of sets $\{S_i : i \in I\}$ and for each $i \geq j \in I$, a function $\pi_{i,j}:S_i \to S_j$ so that $\pi_{i,i}$ is the identity function for each $i$, and $\pi_{i,k} = \pi_{j,k} \circ \pi_{i,j}$.  In the present paper, the $S_i$ will generally be Galois groups of finite extensions.

Let $(S_i, \pi_{i,j})$ be an inverse system in which each $S_i$ is a topological space, and let $S=\varprojlim S_i$.  Further, let $\pi_i: S \rightarrow S_i$ be the restriction to $S$ of the projection from $\prod_jS_j$ to $S_i$.  To define a topology on $S$, we use the sets of the form $\pi_i^{-1}(U_i)$, where $U_i$ is an open subset of $S_i$ as a basis.  When the $S_i$ are finite sets with the discrete topology, as will often be the case in this paper, this induced topology is called the \emph{profinite topology}.  In our case, the $S_i$ will be an inverse system of discrete finite subsets of Galois groups of finite extensions, and the resulting profinite topology is called the \emph{Krull topology}.

\begin{remark}
   From now on all references to open and closed sets will refer to Krull topology, unless some other topology is explicitly identified. 
\end{remark}

The proof of the following lemma can be found in Section 1.1, Chapter 1 of \cite{friedjarden}.
\begin{lemma}
\label{le:basic}
Let $\overline{H} \subset \overline{G}=\Gal(\overline{\Q}/\Q)$ be such that $\overline{H}$ contains all elements of $\overline{G}$ restricting to the set of elements $H \subseteq \Gal(K/\Q)$ for some finite Galois extension $K$ of $\Q$.  In this case, $\overline{H}$ is a basic open subset of $\overline{G}$.   Furthermore, this subset is also closed.
\end{lemma}
We will need to use a smaller class of open sets which will still constitute a basis for Krull topology.  To describe this class, we will use the following notation.

\begin{definition}
    Let $\tau$ be an embedding from an algebraic number field to $\overline{\Q}$.  Then $E(\tau)$ is the set of all extensions of $\tau$ in $\Gal(\overline{\Q}/\Q)$.
\end{definition}

\begin{lemma}
\label{le:smallbasis}
Let $\{\alpha_i\} \subset \overline{\Q}$ be such that $\Q(\alpha_i) \subset \Q(\alpha_{i+1})$ and $\overline{\Q}=\bigcup_i\Q(\alpha_i)$.  Let $T_i$ be the set of all embeddings $\tau: \Q(\alpha_i) \longrightarrow \overline{\Q}$.  Every open subset of $\Gal(\overline{\Q}/\Q)$ is a union of sets of the form $E(\tau_i)$, where $\tau_i$ is an element of $T_i$.
\end{lemma}
\begin{proof}
Let $K/\Q$ be a finite extension, let $\gamma: K\rightarrow \overline{\Q}$ be an embedding and let $E(\gamma)$ be the set of all extensions of $\gamma$ in $\Gal(\overline{\Q}/\Q)$.  We want to show that $E(\gamma)=\bigcup E(\tau)$ for some $\tau \in T_i$ for some $i$.  Let $\alpha_j$ be such that $K \subset \Q(\alpha_j)$.  Let $\gamma_1,\ldots,\gamma_r$ be all the distinct extensions of $\gamma$ to $K(\alpha_j)$.  Then $E(\gamma)=E(\gamma_1)\cup E(\gamma_2) \cup \ldots \cup E(\gamma_r)$.  But each $\gamma_m=\tau$ for some $\tau \in T_j$.  Hence, the assertion of the lemma holds. 
\end{proof}
\begin{lemma}
A one element subset of $\Gal(\overline{\Q}/\Q)$ is closed. 
\end{lemma}
\begin{proof}
By Lemma 1.1.3 of \cite{friedjarden}, the topology of the absolute Galois group is Hausdorff.  In Hausdorff topology a set consisting of one point is closed.  
\end{proof}




We will also make use of the following lemma.

\begin{lemma}
Let $\sigma \in \Gal(\overline{\Q}/\Q)$ and let $K^{\sigma}$ be its fixed field.  Then the closure of the group generated by $\sigma$ is $\Gal(\overline{\Q}/K^{\sigma})$.
\end{lemma}
\begin{proof}
Let $G_{\sigma}$ be the group generated by $\sigma$.  Then $G_{\sigma} \subset \Gal(\overline{\Q}/K^{\sigma})$.  By infinite Galois theory (Proposition 1.3.1 of \cite{friedjarden}) we have that $\Gal(\overline{\Q}/K^{\sigma})$ is closed.  Suppose that closure of $G_{\sigma}$ is a proper subset of $\Gal(\overline{\Q}/K^{\sigma})$.  In this case we have two closed subgroups of $\Gal(\overline{\Q}/\Q)$ corresponding to the same fixed field.  This is impossible by infinite Galois theory (again by Proposition 1.3.1 of \cite{friedjarden}).
\end{proof}
\subsection{Computing Haar Measure of Absolute Galois Groups}


Let $G = \Gal(\overline{\Q}/\Q)$, and let $\calB$ (the Borel field of $G$) be the smallest family of subsets of $G$ containing all closed subsets and closed under taking complements in $G$ and countable unions (hence also intersections).

\begin{definition}[Haar Measure]
A Haar measure on $G$ is a function $\mu: \calB \rightarrow \R$
such that
\begin{itemize}
\item $0 \leq \mu(B) \leq 1$ for all $B \in \calB$,
\item $\mu(\emptyset)=0, \mu(G)=1$,
\item  If $\{B_i\}$ is a sequence of pairwise disjoint Borel sets, then $\mu(\bigcup_i B_i)=\sum_i \mu(B_i)$ ($\sigma$-additivity),
\item If $B \in \calB$ and $g \in G$, then $\mu(gB) = \mu(Bg) = \mu(B)$, and
\item For each $B \in \calB$ and each $\varepsilon >0$ there exist an open set $U$ and a closed set $C$ such that $U\subseteq B \subseteq C$ and $\mu(C\setminus U) <\varepsilon$ (regularity).
\end{itemize}
\end{definition}

By Proposition 18.1.3 and 18.2.1 of \cite{friedjarden}, Haar measure exists and it is unique.
 

\begin{lemma}[A subgroup fixing a finite extension]
\label{inf ext}
If $K$ is a Galois number field, then $\Gal(\overline{\Q}/K)$ is of Haar measure $\frac{1}{[K:\Q]}$.
\end{lemma}
\begin{proof}
$\Gal(\overline{\Q}/K)$ is the set of all extensions of the identity automorphism of $K$, and therefore is a basic open set, and hence a Borel set.  (Of course, it is also closed, by the Fundamental Theorem of Infinite Galois Theory, as mentioned earlier.)  Further, $\Gal(K/\Q)\cong \Gal(\overline{\Q}/\Q)/\Gal(\overline{\Q}/K)$.  Thus the index of the group must be the degree of $K$ over $\Q$, and the measure must be $\frac{1}{[K:\Q]}$ by the invariance of Haar measure.
\end{proof}
Once we established the connection between the measure of the absolute Galois group of a Galois field with the degree of the field, we can now show that the same connection exists for all finite extensions of $\Q$.
\begin{lemma}
\label{le:notGalois}
Let $K$ be any number field.  Then $\Gal(\overline{\Q}/K)$ is of measure $\frac{1}{[K:\Q]}$.
\end{lemma}
\begin{proof}
Let $K^G$ be the Galois closure of $K$ over $\Q$.  Then $\Gal(\overline{\Q}/K)$ is the set of all extensions of automorphisms of $K^G$ contained in the $\Gal(K^G/K)$. If $\sigma \in \Gal(\overline{\Q}/\Q)$ restricts to identity on $K^G$, then $\sigma$ restricts to identity on $K$, and therefore $\Gal(\overline{\Q} /K^G) \subseteq \Gal(\overline{\Q}/K)$.  Further, if $\tau \in \Gal(\overline{\Q} /K)- \Gal(\overline{\Q} /K^G)$, then $\tau$ restricts to a nontrivial automorphism of $\Gal(\overline{\Q} /K^G)$ restricting to identity on $K$. In other words, $\tau$ restricts to a non-trivial element of $\Gal(K^G/K)$.  At the same time if $\tau, \gamma \in \Gal(\overline{\Q}/K)$ restrict to the same element of $\Gal(K^G/K)$, then $\gamma \tau^{-1} \in \Gal(\overline{\Q}/K^G)$.  Hence, $[\Gal(\overline{\Q}/K):\Gal(\overline{\Q} /K^G)]=[K^G:K]$.  Therefore, by additivity of Haar measure we have that $\mu(\Gal(\overline{\Q}/K))=[K^G:K]\mu(\Gal(\overline{\Q} / K^G))$. Therefore, by Lemma \ref{inf ext}, we have that $\mu(\Gal(\overline{\Q}/K))=[K^G:K]\frac{1}{[K^G:\Q]}=\frac{1}{[K:\Q]}$. 
\end{proof}
 \begin{remark}
 Let $K/\Q$ be an infinite algebraic extension.  Let $K=\bigcup_i K_i$, where $K_i \subset K_{i+1}$ and $K_i$ is a number field.  In this case, $\Gal(\overline{\Q}/K)=\bigcap_i \Gal(\overline{\Q}/K_i)$, where each $\Gal(\overline{\Q}/K_i)$ is a clopen set.  Thus, the intersection is closed and hence measurable.  Since $\mu\left(\Gal(\overline{\Q}/K)\right) \leq \mu(\Gal(\overline{\Q}/K_i))$ for every $i$, it follows that  $\Gal(\overline{\Q}/K)$ has measure 0.
 
 \end{remark}

 \subsection{Computable Measures}\label{compmeas}

In discussing a measure in the context of computability, it is reasonable to ask whether, or to what extent, the measure, considered as a function from measurable sets to real numbers, can be regarded as a computable function.  Making sense of this general question is well beyond the scope of this paper.  We can simplify this issue by considering computability of the measure on a distinguished class of sets.

Different choices of this distinguished class of sets can be justified on different grounds.  For example, by Carath\'eodory's Extension Theorem (see, for instance, Theorem 12.8 of \cite{Royden}), if we have a measure defined on an algebra $\mathcal{A}$ of sets (that is, a collection of sets closed under complement, finite union, and finite intersection), then that measure admits a unique extension to the smallest $\sigma$-algebra containing $\mathcal{A}$.  Thus, we might aim to have the restriction of the measure to some algebra of sets be, in some sense, computable.  This is the approach taken by \cite{BienvenuMerkle2007,reimann}.

Another approach is to choose the class of computable measurable sets.  This is the approach taken by \cite{PaulySeonZiegler}.

In this paper, we choose the first option.  We plan, in future work, to explore the second.  Our algebra of sets will be the algebra generated by the basic open sets of the Krull topology.

\begin{lemma}
    Let $K$ be a Galois number field.  Let $\tau \in \Gal(K/\Q)$. Let $\mu$ be the Haar measure.   Then $\mu(E(\tau))=\mu(\Gal(\overline{\Q}/K))$. 
\end{lemma}
\begin{proof}
      If $\tau \in \Gal(K/\Q)$, then let $\tau^*, \tau^{\#} \in E(\tau)$.  Then $\tau^*\left(\tau^{\#}\right)^{-1}\upharpoonright_K = \id_K$.
    Therefore, $\tau^*\left(\tau^{\#}\right)^{-1} \in E(\id_K)$.  Hence,
    $\tau^* \in E(\id_K)\left(\tau^{\#}\right)^{-1}=\Gal(\overline{\Q}/K)\left(\tau^{\#}\right)^{-1}$. So $E(\tau)=\Gal(\overline{\Q}/K)(\tau^{\#})^{-1}$. Thus, $\mu(E(\tau))=\mu(\Gal(\overline{\Q}/K)(\tau^{\#})^{-1})=\mu(\Gal(\overline{\Q}/K))$.
\end{proof}
 \begin{corollary}
    Let $\{K_1,\ldots, K_n\}$ be a finite collection of number fields.  Let $\tau_i \in \Gal(K_i/\Q)$.  Then there is an effective procedure to compute the Haar measure of $\bigcup\limits_{i=1}^n E(\tau_i)$.
 \end{corollary}
 \begin{proof}
     By induction, it is enough to show how to compute the Haar measure of a union of two sets $E(\tau_1)$ and $E(\tau_2)$.  Let $K$ be any Galois number field containing $K_1$ and $K_2$.  Let $\lambda_1, \ldots, \lambda_r$ be all the extensions of $\tau_1$ to $\Gal(K/\Q)$.  Similarly, let $\theta_1, \ldots,\theta_m$ be all the extensions of $\tau_2$ to $\Gal(K/\Q)$.  Then $E(\tau_1)=\bigcup_{i=1}^r E(\lambda_i)$ and $E(\tau_2)=\bigcup_{j=1}^m E(\theta_j)$.  Observe that $E(\lambda_i) \cap E(\lambda_j)=\emptyset$ for $i \ne j$.  Similarly, $E(\theta_i)\cap E(\theta_j)=\emptyset$ for $i\ne j$. Let $\{\nu_1,\ldots,\nu_l\}=\{\lambda_1,\ldots,\lambda_r\}\cap \{\theta_1,\ldots,\theta_m\}$.  Finally let $d=[K:\Q]$.  Then the measure 
     \[
     \mu\left(E(\tau_1)\cup E(\tau_2)\right)=\frac{r+m-l}{d}.
     \]
 \end{proof}
It is clear that a similar procedure will effectively compute the measure of a finite intersection of basic open sets and the complements.  So there is an effective way of computing the Haar measure of every set in the algebra.

\section{Computability Theory of Absolute Galois Groups}\label{SecCompGalois}

Because absolute Galois groups contain maps with countable domain we will use the discussion in Section \ref{SecBaire} to define computable and c.e.\ subsets of $\Gal(\overline{\Q}/\Q)$. All methods for representing functions in Section \ref{SecBaire} rely on a version of a graph of the set of functions under consideration.  At this point, we take a closer look at different versions of such graphs.   Below we list what seem to us the four most natural options for describing absolute Galois groups.
\begin{enumerate}
\item The graph consists of pairs, $\left(\alpha,\beta\right)$ where $\beta$ is one image of $\alpha$ under the action of the absolute Galois group, and $\alpha \in \overline{\Q}$.
\item The graph consists of pairs, $\left(\vec{\alpha},\vec{\beta}\right)$ where $\vec{\beta}$ is one image of $\vec\alpha$ under the action of the absolute Galois group, and $\vec{\alpha} \in \overline{\Q}^n$ for some positive integer $n$.
\item The graph consists of sequences $(\alpha,\beta_1, \dots, \beta_n)$, where $\beta_1, \dots, \beta_n$ are \emph{all possible images} of $\alpha$ under the action of the absolute Galois group, and $\alpha \in \overline{\Q}$.
\item The graph consists of sequences $(\vec{\alpha},\vec{\beta}_1, \dots, \vec{\beta}_n)$, where $\vec{\beta}_1, \dots, \vec{\beta}_n$ are \emph{all possible images} of $\vec{\alpha}$ under the action of the absolute Galois group, and $\vec{\alpha} \in \overline{\Q}^n$ for some positive integer $n$.
\end{enumerate}
Thus, we have four possibilities for how to represent absolute Galois groups.  It would at first seem, for instance, that possibility 2 would contain more information than possibility 1, because it also includes information about consistency of images, and it would seem that possibility 4 would contain the most information.  However, by working with a computable sequence generating $\overline{\Q}$ over $\Q$, and using complements of the graphs we will conclude that all versions of the graph are equivalent for some of our purposes (see Lemma \ref{GaloisGraphTuples}).

 In Lemma \ref{GraphEqStrGraph}, we will show that all four sets are Turing equivalent.  However, some differences arise at the level of enumeration reducibility, as we will see (Section \ref{ceagg}).

\subsection{Computable Absolute Galois Groups}

Here we remind the reader that by a computable field we mean a field computable within a fixed computable algebraic closure of $\Q$.  

\begin{lemma}
Let $K \subset \overline{\Q}$ be a computable field.  Then there exists a computable sequence $\{\alpha_i: i \in \mathbb{N}\}$  such that $K=\Q(\{\alpha_i: i \in \mathbb{N}\})$.
\end{lemma}
\begin{proof}
Given an element $\alpha \in \overline{\Q}$ we can determine whether $\alpha \in K$.  Let $\alpha_1 \in K- \Q$ be such an element with the smallest code and let $K_1=\Q(\alpha_1)$.  (Note that this step is effective since $K$ is computable.)  By Lemma \ref{le:compext} we have that $K_1$ is also computable.  Thus proceeding inductively we construct a computable sequence $\alpha_i$ such that $\bigcup_{i=1}^{\infty}\Q(\alpha_1,\ldots, \alpha_i)=K$.
\end{proof}
\begin{lemma}
If $K$ is computable, then there exists a computable set of elements $\{\alpha_i: i \in \mathbb{N}\}$ of $\, \overline \Q$ such that $K(\{\alpha_i: i \in \mathbb{N}\})=\overline{\Q}$.  Further the sequence can be selected so that for each $j \in \mathbb{N}$ we have that $K(\alpha_0,\ldots,\alpha_j)$ is Galois over $K$.
\end{lemma}
\begin{proof}
The sequence $\alpha_i$ can be constructed inductively in the following  fashion.  Find $\beta_0 \in \overline{\Q}- K$ with the smallest code.  Next find all conjugates of $\beta_0$ over $\Q$. (This is an effective step since $K$ is computable.) Let $N_0$ be the extension  of $K$ obtained by adjoining all conjugates of $\beta_0$ to $K$. Then $N_0/K$ is Galois.   Assume inductively that we have constructed $N_i$ such that $N_i/K$ is Galois.  Now we find $\beta_i \not \in N_i$ with the smallest code and find all of its conjugates over $K$ and adjoin them to $N_i$.  The resulting field $N_{i+1}$ is still Galois over $\Q$. 

Since we always select an element not in $N_i$ with the smallest code, every element of $\overline{\Q} - K$ will eventually be included at some step of the construction.  Thus,  $\overline{\Q}=\bigcup_{i=0}^{\infty}N_i$.

\end{proof}

\begin{definition}[The graph and strong graph of a subset of an absolute Galois group]
Let $S$ be a subset of $\Gal(\overline{\Q}/\Q)$.
\begin{enumerate}
    \item The \emph{graph of $S$}, denoted $\Gamma(S)$, is the set of pairs of the form $(\alpha,\alpha_1)$, where every element of $\overline{\Q}$ appears as $\alpha$, and $\alpha_1$ is an image of $\alpha$ under the action of $S$.  We denote by $S(\alpha)$ the set $\{\sigma (\alpha) : \sigma \in S\}$.
    \item Let $S$ be a subset of $\Gal(\overline{\Q}/\Q)$.  The \emph{strong graph of $S$}, denoted $\Gamma_+(S)$, is the set of tuples of the form $(\vec{\alpha},\vec{\alpha}_2,\ldots , \vec{\alpha}_{k_{\alpha}})$, where every tuple of $\overline{\Q}$ appears as $\vec{\alpha}$ and $\vec{\alpha}_2, \dots, \vec{\alpha}_{k_\alpha}$ constitute a list (in order of increasing index) of \emph{all} images of $\vec{\alpha}$ under the action of the subgroup. 
    \end{enumerate}
\end{definition}

\begin{lemma}\label{parttowhole1}
Let $S$ be a closed subset of $G(\overline{\Q}/\Q)$.  Let $\tau \in G(\overline{\Q}/\Q)$ be such that for every $\alpha \in \overline{\Q}$ we have that $\tau(\alpha) \in S(\alpha)$.  Then $\tau \in S$.
\end{lemma}

\begin{proof}
Let $\{\alpha_i: i \in \mathbb{N}\} \subset \overline{\Q}$ be such that $\Q(\alpha_i) \subset \Q(\alpha_{i+1})$ and $\overline{\Q}=\bigcup\limits_{i\in \mathbb{N}}\Q(\alpha_i)$. Let $T_i$ be the collection of all embeddings of $\Q(\alpha_i)$ into $\overline{\Q}$.   Suppose $\tau \not \in S$.  Then $\tau \in S^c$, where $S^c$, the complement of $S$ in $\Gal(\overline{\Q}/\Q)$, is open.  

By Lemma \ref{le:smallbasis}, we know that the sets of the form $E(\nu)$  --- that is, extensions of a single embedding $\nu:\Q(\alpha_i) \to \overline{\Q}$ for some $i$, constitute a basis for Krull topology.  Therefore, for some collection $\Sigma=\{\tau_{i,j} : i,j \in \mathbb{N}\}$ where each $\tau_{i,j} \in T_i$ for some $i$,  we have that $S^c=\bigcup E(\tau_{i,j})$. Hence, if $\tau \in S^c$,  then for some $i, j$ we have that $\tau_{i,j} \in \Sigma$ and $\tau_{|\Q(\alpha_i)}=\tau_{i,j}$ or equivalently $\tau(\alpha_i)=\tau_{i,j}(\alpha_i)$. 

At the same time, for any $\mu \in S$ we have that $\mu$ cannot restrict to any $\tau_{i,j} \in \Sigma$ because $E(\tau_{i,j}) \subset S^c$.  In other words, $\mu(\alpha_i) \ne \tau_{i,j}(\alpha_i)$.  Therefore, if $\tau \in S^c$, we have that $\tau(\alpha_i) \not \in S(\alpha_i)$ contradicting our assumptions on $\tau$.
\end{proof}

We will frequently apply Lemma \ref{parttowhole1} in the case where $S$ is the absolute Galois group of a field.


\begin{lemma}\label{GaloisGraphTuples}
Let $S \subseteq\Gal(\overline{\Q}/\Q)$.  Let $(\alpha_1,\ldots,\alpha_k) \in \overline{\Q}^k$.  Let $A_S$ be the set of $n$-tuples $(\beta_1,\ldots,\beta_n)$ such that there exists $\sigma \in S$ satisfying $\sigma(\alpha_i)=\beta_i$ for $i=1,\ldots, n$.  Then $\Gamma(S) \geq_T A_S$.
\end{lemma}
\begin{proof}
Using Lemma \ref{FindSingleGen}, we can effectively find an element $\gamma \in \overline{\Q}$ such that $\alpha_i \in \Q(\gamma)$ for all $i=1,\ldots, n$.  Using $\Gamma(S)$, we determine all possible images of $\gamma$ under the action of elements of $S$.  Each potential image of $\gamma$ will determine the image of the $n$-tuple $\alpha_1,\ldots,\alpha_n$.

\end{proof}
\begin{lemma}\label{GraphEqStrGraph} Let $S \subseteq\Gal(\overline{\Q}/\Q)$.  Then $\Gamma(S) \equiv_T \Gamma_+(S)$.
\end{lemma}

\begin{proof} By letting $\vec{\alpha}$ range over singletons, it is clear that $\Gamma(S) \leq_T \Gamma_+(S)$.  Since a Turing oracle for $\Gamma(S)$ gives information not only about elements of $\Gamma(S)$, but also about elements of its complement, the converse follows immediately from Lemma \ref{GaloisGraphTuples}.
\end{proof}

\begin{proposition}
\label{prop:sg}
    If $G \subseteq \Gal(\overline{\Q}/\Q)$ and $\Gamma_+(G)$ is computably enumerable, then $\Gamma_+(G)$ is computable (and consequently $\Gamma(G)$ is also computable).
\end{proposition}

\begin{proof}
    Suppose we are given an enumeration of $\Gamma_+(G)$.  For every $\alpha \in \overline{\Q}$, the listing of $\Gamma_+(G)$ contains exactly one tuple of the form $(\alpha, \beta_1,\ldots,\beta_r),$ where $\beta_1,\ldots,\beta_r$ are all the possible images of $\alpha$ under $G$.  Therefore, the enumeration of $\Gamma_+(G)$ will eventually list the tuple corresponding to $\alpha$.  Since all elements of $\Gamma_+(G)$ are of the form above, we can effectively answer the question whether any tuple of the form $(\gamma,\delta_1,\ldots,\delta_s) \in \Gamma_+(G)$.
\end{proof}

Of course, it is possible that two sets can be Turing equivalent while one is c.e.\ and the other is not.  For instance, let $S$ be any set which is computably enumerable but not computable.  Then $S \equiv_T S^c$, but $S^c$ is not enumerable.  However, we believe that the relationship between $\Gamma(G)$ and $\Gamma_+(G)$ is a natural example of this phenomenon.

As we show below, $\Gamma_+(G)^c \leq_e \Gamma(G)^c$.  That is, given an enumeration of the complement of $\Gamma$, we can effectively produce an enumeration of the complement of $\Gamma_+$.  
Indeed, let $S \subset \Gal(\overline{\Q}/\Q)$ and  let $\bar x=(\alpha,\beta_1, \ldots,\beta_r)\in \overline{\Q}^{r+1}$. Then  $\bar x \not \in \Gamma_+(S)$ for one of two reasons.  Either some $\beta_i$ is not a conjugate of $\alpha$ over $\Q$ (and we can effectively determine that) or no element of $S$ sends $\alpha$ to $\beta_i$.  In the last case the pair $(\alpha,\beta_i)$ will appear in the complement of $\Gamma(S)$.  So a listing of the complement of $\Gamma(S)$ will produce a listing of the complement of $\Gamma_+(S)$.

Since, for any $S$, the graph of $S$ is contained in the set of all finite sequences of elements of $\overline{\Q}$, it follows that the graph of any subset of the absolute Galois group of $\Q$ is countable.

From the discussion in Section \ref{SecBaire} it follows that the following theorem holds.

\begin{theorem}\label{EqDnf:Computable} The following conditions on a subset $S$ of $\Gal(\overline{\Q}/\Q)$ are equivalent:
\begin{enumerate}
    \item For each $n \geq 1$ there is a computable function $I_n:\overline{\Q}^n \to \mathbb{N} \times \overline{\Q}^{<\omega}$ such that $I_n(\vec{\alpha}) = (m_{\vec{\alpha}},T_S)$
  if and only if $T_S$ is a sequence of length $m_{\vec{\alpha}}$, and is exactly a sequence of
  images of $\vec{\alpha}$ under elements of $S$.
  \item $\Gamma_+(S)$ is computable.
  \item $\Gamma(S)$ is computable.
  \item There is a computable function (as described in Section \ref{SecBaire})
\[
f_S(d,r):=\left\{\begin{array}{ll} 1 & \mbox{if $B_r(d) \cap {(^\sigma S)} \neq
  \emptyset$}, \\ 0 & \mbox{if $B_{2r}(d) \cap ({^\sigma S}) = \emptyset$}, \\ 0\mbox{
  or } 1 & \mbox{otherwise}\\ \end{array}\right.\]
where 
\begin{enumerate}
    \item $d \in \omega^{<\omega}$,
    \item $r$ is a rational number, and
    \item $B_r(d)$ is the ball of radius $r$ about $d$ in Baire space.
\end{enumerate}
\end{enumerate}
\end{theorem}
 \begin{remark} 
  Theorem \ref{EqDnf:Computable} also holds when ``computable" is replaced by ``computable relative to $X$" for any $X \subseteq \omega$ since the proof relativizes.  
  \end{remark}

We now define computability in the natural way.

\begin{definition}\label{dfn:computable} We say that $S \subseteq \Gal(\overline{\Q}/\Q)$ is computable if and only if it satisfies the equivalent conditions of Theorem \ref{EqDnf:Computable}.\end{definition}


\begin{definition}[Automorphism tree of the absolute Galois group of a field]
Let $K$ be a subfield of $\overline{\Q}$.  Let $K \subset F_1 \subset \cdots$ be a tower of fields such that $\bigcup_{i=1}^\infty F_i=\overline{\Q}$ and $F_{i+1}/F_i$ is finite for all $i \geq 1$, and such that the sequence $(F_i : i \in \mathbb{N})$ is uniformly computable in $K$ (See Definition \ref{UnifComp}).   Consider a tree of $\Gal(\overline{\Q} /K)$ constructed in the following fashion.  Let identity on $K$ be the root of the tree.  The level $i$ of the tree will contain all extensions of the identity on $K$ to embeddings of $F_i$ into $\overline{\Q}$, and if $\tau$ on level $i+1$ is a child of $\mu$ on level $i$, then $\mu$ corresponds to an embedding of $F_i$ into $\overline{\Q}$ and $\tau$ is an embedding of $F_{i+1}$ restricting to $\mu$ on $F_i$.
\end{definition}
\begin{proposition}
\label{prop:path}
Every path in the automorphism tree of a field $K$ corresponds to an element of the absolute Galois group of $K$.
\end{proposition}
\begin{proof}
Let $\tau_0=\id, \tau_1,\ldots$ be a path through the automorphism tree of the absolute Galois group of an algebraic field $K$.  By construction $\tau_i$ is an embedding of $F_i$ into $\overline{\Q}$ keeping $K$ fixed and $\tau_i$ restricts to $\tau_{i-1}$ on $F_{i-1}$.  We show that the path defines an automorphism $\tau$ of $\overline{\Q}$ fixing $K$.  Let $\alpha \in \overline{\Q}$.  Then by construction of $\{F_i\}$, we have that for some $j$ the element $\alpha \in F_j$.  Thus $\tau_k(\alpha)$ for $k\geq j$ is defined and $\tau_k(\alpha)=\tau_r(\alpha)$ for any $r,k \geq j$.  We set $\tau(\alpha)=\tau_j(\alpha)$.  

Suppose $\beta \in \overline{\Q}$.  Now $\beta$ has finitely many conjugates over $K$, and they are all contained in some $F_i$, where $\tau_i$ must permute them.  Thus, $\beta$ is in the range of $\tau$.  
The function $\tau$ is clearly an injective homomorphism and by the argument above it is surjective.  Thus, it is an automorphism of $\overline{\Q}$ keeping $K$ fixed.
\end{proof}


\begin{proposition}\label{intersectinggroups}
Let $G_1, G_2$ be two absolute Galois groups (of some fields, but these fields will play no explicit role in the statement or proof). Let $\{\alpha_i: i \in \mathbb{N}\}$ be such that $\Q(\alpha_i) \subset \Q(\alpha_{i+1})$ and $\bigcup_{i\in \omega}\Q(\alpha_i)=\overline{\Q}$. Then $G_1 \cap G_2 \not =\{\id\}$ if and only if  there exists a sequence $\bar \gamma=\{\gamma_i\} \subset \overline{\Q}$ and a collection $\{T_{\bar{\gamma},i}\} \subset \Gal(\overline{\Q}/\Q)$  such that
\begin{enumerate}
    \item  for all $i$ we have that $\gamma_i \in G_1(\alpha_i) \cap G_2(\alpha_i) \subset \overline{\Q}$, 
    \item  for all $i$, we have $T_{\bar{\gamma},i-1}=\{\tau \in G_1 \cap G_2: \tau(\alpha_{i-1})=\gamma_{i-1}\}$, 
\item for all $i$ we have $\gamma_i \in T_{\bar \gamma,i-1}(\alpha_i)$, and
\item for all but finitely many $i$ we have $\gamma_i \neq \alpha_i$.
\end{enumerate} 
\end{proposition}

\begin{proof}
If $\sigma \in \Gal(\overline{\Q}/\Q)$ and $\sigma \neq \id$, then there is a greatest $i$ such that $\sigma(\alpha_i) = \alpha_i$.  Otherwise, let $\{i_j\}$ be a sequence of indices such that $\sigma(\alpha_{i_j})=\alpha_{i_j}$.  Since $\bigcup \Q(\alpha_{i_j})=\overline{\Q}$, we have that $\sigma$ does not move any element of $\overline{\Q}$ and therefore is equal to identity.  

First assume there exists $\sigma \in G_1 
\cap G_2$ with $\sigma  \ne \id$.  Then let $\gamma_i=\sigma(\alpha_i)$.  This satisfies the first requirement.  By the discussion above $\gamma_i \ne \alpha_i$ for all but finitely many $i$. So the sequence $\bar \gamma$ satisfies  the last requirement. Next define $T_{\bar \gamma, i}$ to satisfy the second requirement and observe that
by construction of $\bar \gamma$ we have that $\sigma \in T_{\bar \gamma,i-1}$.  Therefore, $\gamma_i \in T_{\bar \gamma,i-1}(\alpha_i)$ satisfying the third requirement.

Conversely, suppose  there exists a sequence $\bar \gamma \subset \overline{\Q}$ and a collection $\{T_{\bar \gamma,i}\} \subset \Gal(\overline{\Q}/\Q) $ satisfying all the requirements above.    It is enough to show that there exists an automorphism $\sigma \in \Gal(\overline{\Q}/\Q)$ such that $\sigma(\alpha_i)=\gamma_i$, since by Lemma \ref{parttowhole1} such an automorphism $\sigma \in G_1 \cap G_2$.  Further, the last requirement on the sequence $\{\gamma_i\}$ implies that $\sigma \ne \id$.

We define the automorphism $\sigma$ inductively.  Let $\sigma_0=\id$.  Assume we have defined $\sigma_{i-1}:\Q(\alpha_{i-1}) \longrightarrow \overline{\Q}$ by setting $\sigma_{i-1}(\alpha_{i-1})=\gamma_{i-1}$ and let $\sigma_i(\alpha_i)=\gamma_i$.  We claim that the sequence $\{\sigma_i\}$ is a path through an automorphism tree of $\Q$, where $F_i=\Q(\alpha_i)$.  In other words, we claim that $\sigma_{i|F_{i-1}}=\sigma_{i-1}$.

By assumption there exists $\tau \in G_1 \cap G_2$ such that $\tau(\alpha_{i-1})=\gamma_{i-1}$ and $\tau(\alpha_i)=\gamma_i$.  Therefore, if we set $\sigma_i=\tau_{|\Q(\alpha_i)}$ we can conclude that $\sigma_i: \Q(\alpha_i)\longrightarrow \overline{\Q}$ is an embedding and $\sigma_i{|_{\Q(\alpha_{i-1})}}=\sigma_{i-1}$.  Hence, by the definition of an automorphism tree, we have that $\{\sigma_i\}$ is a path.  Thus, by Proposition \ref{prop:path} we have that there exists $\sigma \in \Gal(\overline{\Q}/\Q)$ such that $\sigma(\alpha_i)=\gamma_i$.
\end{proof}

\subsection{Computably Enumerable Absolute Galois Groups}\label{ceagg}

Having defined computable absolute Galois groups, we now proceed to the more difficult situation of computably enumerable absolute Galois groups.  In formulating definitions of c.e.\ absolute Galois groups, we would like to preserve some algorithmic parity between the graph and the strong graph.  One difficulty is that while computability of the graph and the strong graph  ultimately give the same information, the same is not true for computable enumerability.  Indeed, Proposition \ref{prop:sg}, in combination with Lemma \ref{GraphEqStrGraph}, shows that if the graph is computably enumerable but not computable, then the strong graph is not even computably enumerable and, as we discussed above, the enumeration relation connects the complements of the graph and the strong graph.

Since the connection of enumerability is between the complement of the graph and the complement of the strong graph, we adopt the following definition.

\begin{definition}[C.e. Galois groups]
Let $G$ be an absolute Galois group.  Then we say that $G$ is c.e.\ if the complement of its graph is c.e..
\end{definition}

The additional reason for using the complement of the graph, instead of the graph itself, is that to enumerate the field, we would need to know the complement of the graph.
Since the focus of the present paper is on random fields, and not on random groups, we believe that the correct location of the enumerability is in the fields.  One alternative was to require \emph{both} the graph of the group and the field to be computably enumerable, which would collapse enumerability to computability, but does not seem to change many of the results of the present paper.

In general, we will refer to an absolute Galois group as having some algorithmic property (enumerability, computability relative to an oracle, etc.) if and only if the complement of its graph has this property.

Observe that if $G$ is computable, then both its graph and its complement are c.e..

\begin{proposition}\label{FieldsAndGaloisTeq} There is a Turing functional which, given the characteristic function of a subfield $K$ of a fixed computable algebraic closure $\overline{\Q}$, will compute the characteristic function of the graph of $\Gal(\overline{\Q}/K)$, and a Turing functional which will, given the characteristic function of the graph of a closed subgroup of $\Gal(\overline{\Q}/\Q)$, compute the characteristic function of its fixed field.
\end{proposition}
\begin{proof}  
Using the characteristic function of $\Gamma(\Gal(\overline{\Q}/K))$, we can determine the set of all $x \in \overline{\Q}$ which are fixed by $\Gal(\overline{\Q}/K)$; that is, the elements of $K$.


We now show that $\Gamma(\Gal(\overline{\Q}/K)) \leq_T K$.
Using Lemma \ref{FindConjugates}, we determine all the conjugates of $x$ over $K$: $x=x_1, \ldots, x_r$.  Then the pairs $(x,x_i)$ are the only pairs from $\Gamma(\Gal(\overline{\Q}/K))$ that have $x$ as its first component.

\end{proof}

\begin{lemma}\label{EnumeratingGroups} There is a $\emptyset'$-enumeration of the computably enumerable absolute Galois groups.
\end{lemma}


\begin{proof} Recall that we identify a group with its graph, and a group is called computably enumerable just in case the complement of its graph is computably enumerable.  The oracle $\emptyset'$ can enumerate (indeed, compute) not only the computably enumerable sets, but their complements, as well.  There is, then, a $\emptyset'$-enumeration of the co-computably enumerable sets $S$ of pairs $(\alpha_i,y)$ where $\{\alpha_i : i \in \mathbb{N}\}$ are as in Lemma \ref{parttowhole1}.  In other words, we enumerate the set of co-computably enumerable sets of the right type to be absolute Galois groups.

It remains to sieve the enumeration to list only genuine absolute Galois groups.  To this end, using $\emptyset'$, we will check, for each $i$, whether the images for $\alpha_i$ given are consistent with the images given in for $\alpha_j$ with $j<i$.  We also check whether the set of partial automorphisms specified by restriction to $\Q(\alpha_i)$ constitutes a group under composition.  For each $i$, there are only finitely many conditions to check.  This procedure allows $\emptyset'$ to enumerate the computably enumerable absolute Galois groups, as required.
\end{proof}



\section{Randomness}\label{SecRand}



\subsection{Defining Random Fields}

In the definition that follows, the Martin-L\"of tests are made against closed subgroups.  In definitions of random reals, we do not consider any group structure, and do not use closed sets.  However, here the test should be limited to subgroups that correspond to subfields of $\overline{\Q}$, and these are the closed subgroups.

\begin{definition} Let $\mu$ be the normalized Haar measure.
\begin{enumerate}
    \item A $\mu$-test is a uniformly computably enumerable sequence $(S_i: i \in \mathbb{N})$ of closed subgroups of $\Gal(\overline{\Q}/\Q)$ such that $\mu(S_i) < 2^{-i}$.
    \item We say that $\sigma \in\Gal(\overline{\Q}/\Q)$ is \emph{random} if for any $\mu$-test $(S_i: i \in \mathbb{N})$ we have $\sigma \notin \bigcap\limits_{i \in \mathbb{N}} S_i$.
    \item We say that an algebraic field is \emph{random} if and only if it is an infinite extension of $\rat$ and its absolute Galois group contains a random element.
\end{enumerate}
\end{definition}

We recall that when we describe a uniformly computably enumerable sequence of subgroups, we mean that the sequence of complements of the graphs of those groups is uniformly computable.

Again, the choice of definition is not obvious.  We might, from an algebraic perspective, be led to the following alternate definition.

\begin{definition} Let $\mu$ be the normalized Haar measure.
\begin{enumerate}
\item An element of $G = \Gal(\overline{\mathbb{Q}}/\mathbb{Q})$ is said to be \emph{not Haar random} if it is contained in a computably enumerable subgroup of $G$ of measure 0.
\item An algebraic field is said to be \emph{Haar random} if and only if it is an infinite extension of $\rat$ and its absolute Galois group contains a Haar random element.
\end{enumerate}

\end{definition}

We should note that this definition of Haar randomness is reminiscent of the standard definition of "weak $1$-randomness," while the definition of random in the previous definition corresponds more closely with 1-randomness.  Weak 1-randomness is a strictly weaker condition on real numbers than 1-randomness.  However, this distinction collapses in our context.

\begin{proposition} An element $\sigma \in \Gal(\overline{\Q}/\Q)$ is Haar random if and only if it is random.  Consequently, an algebraic field is Haar random if and only if it is random.
\end{proposition}

\begin{proof}  Suppose that $H$ is a c.e.\ group of measure zero containing $\sigma$.  Then we may take $S_i = H$ for all $i \in \mathbb{N}$, showing that $\sigma$ is not random.  Suppose now that $\sigma$ is not random but is Haar random.  Then there exists a sequence of  of c.e.\ absolute Galois groups such that $\mu(G_i)<2^{-i}$ and $\sigma \in \bigcap\limits_{i \in \mathbb{N}} G_i$.  If $\mu(G_i) =0$, for some $i$, then $\sigma$ is not Haar random and we have a contradiction.  

Assume now that the measures of all groups are positive.  Let $H=\bigcap\limits_{i \in \mathbb{N}} G_i$.  Then $\mu(H)=0$, and the complement of $\Gamma(H)$ is the union of the complements of the $\Gamma(G_i)$.  We want to show that $H$ is c.e.\ (that is, that the complement of the graph of $H$ is c.e.) to obtain a contradiction. Note that, since the definition of a $\mu$-test required a uniform sequence (that is, a uniform enumeration of the complements of the $\Gamma(G_i)$), we know that the sequence $\left(\Gamma(G_i)^c : i \in \mathbb{N}\right)$ is uniformly computably enumerable.  It follows that the union of this sequence is computably enumerable, and $H$ is computably enumerable.\end{proof}

\subsection{Properties of Random Fields}

The following property is related to the concept of an immune set.

\begin{proposition}
Let $K$ be an infinite extension of $\rat$.  If $K$ contains an c.e.\ subfield of infinite degree, then $K$ is not random.
\end{proposition}

\begin{proof}
Let $K^-$ be an infinite degree c.e.\ field contained in $K$.  Then $\Gal(\overline{\Q}/K) \subset \Gal(\overline{\Q}/K^-),$ while $\Gal(\overline{\Q}/K^-)$ is c.e.\ and of measure zero.
\end{proof}

It would be tempting to think that this means that random fields are immune.  However, they do include the infinite c.e.\ set of rationals.  We could, however, define another property related to immunity, that a field contain no c.e.\ subfield of infinite degree.

\begin{lemma}
Let $\sigma \in \Gal(\overline{\Q}/\Q)$.  Then $\sigma$ is random if and only if the fixed field of $\sigma$ is of infinite degree and contains no c.e. subfields.
\end{lemma}
\begin{proof}
If the fixed field of $\sigma$ contains an infinite degree c.e. field $K$,  then $\Gal(\overline{\Q}/K)$ is of measure 0 and c.e..  Further, $\sigma \in \Gal(\overline{\Q}/K)$.  Thus, $\sigma$ is not random.  Conversely, suppose the fixed field of $\sigma$ does not contain any c.e.\ subfields, but $\sigma$ is not random.  Then $\sigma \in \Gal(\overline{\Q}/K)$ for some c.e.\ infinite degree field $K$ and we have a contradiction.
\end{proof}
\begin{corollary}\label{nonrandombyfixedfield}
Let $G$ be a closed subgroup of $\Gal(\overline{\Q}/\Q)$; that is, $G$ is a Galois group of some field. Then $G$ is not random if and only if the fixed field of every non-trivial element of $G$ contains a c.e.\ field.
\end{corollary}

\begin{corollary}
For any Turing degree $\mathbf{d}$ there is a field of degree $\mathbf{d}$ which is not random.
\end{corollary}

\begin{proof} Let $F_0$ be generated over the rationals by the $2^q$th roots of $2$ (where $q$ ranges over all natural numbers), and $X \in \mathbf{d}$.  Note that $F_0$ is computably enumerable.  Let $(p_i: i \in \mathbb{N})$ be the rational primes.  We can then build an algebraic extension $F_{\mathbf{d}}$ of $F_0$ which includes $\sqrt[(p_{2i+1})]{p_{2i+1}}$ if and only if $i \in X$ and $\sqrt[(p_{2i+2})]{p_{2i+2}}$ if and only if $i \notin X$.  Now $F_{\mathbf{d}}$ will have degree $\mathbf{d}$, and will contain a computably enumerable infinite extension $F_0$ of $\Q$.
\end{proof}

\begin{corollary}
There is an algebraic field of infinite degree which is not random.
\end{corollary}

\begin{proof}
By Corollary \ref{nonrandombyfixedfield}, any computably enumerable algebraic field of infinite degree will suffice.
\end{proof}
To prove the existence of random elements in $\Gal(\overline{\Q}/\Q)$, we need the following lemmas.
\begin{lemma}
\label{le:intersection}
    Let $L$ be a subfield of $\overline{\Q}$ and let $x \in \overline{\Q} \setminus L$.  Let $\sigma$ be an embedding of $L$ into $\overline{\Q}$.  Then there is an extension of $\sigma$ to $L(x)$ such that $\sigma(x) \ne x$.
\end{lemma}
\begin{proof}
    Let $P(T)$ be the monic irreducible polynomial of $x$ over $L$. Then $\deg P(T)\geq 2$.  If $\sigma(P)=P$, then we can set $\sigma(x)$ to be any root of $P(T)$ not equal to $x$.  If $\sigma(P) \ne P$, then we can set $\sigma(x)$ to be any root of $\sigma(P)$.
\end{proof}
\begin{lemma}
\label{le:notmove}
    Let $L$ be a finitely generated extension of $\Q$.  Let $\sigma$ be an embedding of $L$ into $\overline{\Q}$.  Then there exists infinitely many $x \in \overline{\Q}\setminus L$ such that $\sigma$ can be extended to $L(x)$ by setting $\sigma(x)=x$.
\end{lemma}
\begin{proof}
    Without loss of generality assume that $L$ is Galois over $\Q$ so that every irreducible polynomial over $\Q$ either remains prime over $L$ or splits completely.  Then all polynomials irreducible over $\Q$ of degree prime to $[L:\Q]$ will remain prime over $L$.  Let $x \in \overline{\Q}$ be such that its monic irreducible polynomial over $\Q$ remains prime over $L$. Then $\Q(x)$ and $L$ are linearly disjoint over $\Q$.  If $\alpha \in 
    \overline{\Q}$ is such that $L=\Q(\alpha)$ then the monic irreducible polynomials of $\alpha$ over $\Q$ and $\Q(x)$ are the same.  Therefore, there exists an embedding $\tau$ of $L(x)$ into $\overline{\Q}$ such that $\tau$ is the identity on $\Q(x)$ and $\tau(\alpha)=\sigma(\alpha)$.
\end{proof}

\begin{proposition}[Existence of random elements]\label{ExRandomElts}
There exist a continuum of random elements $\sigma \in \Gal(\overline{\Q}/\Q)$.
\end{proposition}
\begin{proof}
Let $\{K_i: i \in \mathbb{N}\}$ be an enumeration (perhaps not computable) of all infinite degree c.e.\ subfields of $\overline{\Q}$.

Stage 1: Pick an element $x_1 \in K_1 - \Q$.  Define $\sigma(x_1)=y_1,$ where $y_1 \ne x_1$ is a conjugate of $x_1$ over $\Q$. Set $L_1$ to be the Galois closure over $\Q$ of $\Q(x_1)$, and extend $\sigma$ to $L_1$ in the natural way.

Stage 2: Pick an element $x_2 \not \in L_1$ such that $L_1$ is linearly disjoint from $\Q(x_2)$ over $\Q$ and set $\sigma(x_2)=x_2$. (Such an $x_2$ exists by Lemma \ref{le:notmove}.)   Set $L_2$ to be the Galois closure of  $\Q(x_1,x_2)$ over $\Q$, and extend $\sigma$ to $L_2$ in the natural way.

Stage $2n+1$:  Let $K_m$ be an infinite degree c.e.\ field with the smallest index that has not appeared in the construction so far.   Assume inductively that $\sigma$ is defined for elements of  $L_{2n}=\Q(x_1,\ldots,x_{2n})$.  Find an element $x_{2n+1} \in K_m - L_{2n}$.  This can always be done because $K_m$ is not finitely generated. By Lemma \ref{le:intersection} there is an extension of $\sigma$ to $L_{2n}(x_{2n+1})$ such that $\sigma(x_{2n+1})\ne x_{2n+1}$. Set $L_{2n+1}$ to be the Galois closure of $L_{2n}(x_{2n+1})$ over $\Q$, extending $\sigma$ to $L_{2n+1}$ as before.

Stage $2n+2$: Pick an element $x_{2n+2} \in \overline{\Q} - L_{2n+1}$ so  that there is an extension of $\sigma$ to $L_{2n+1}(x_{2n+2})$ such that $\sigma(x_{2n+2})=x_{2n+2}$. (We can find such an element by Lemma \ref{le:notmove}.) Set $L_{2n+2}$ to be the Galois closure of $L_{2n+1}(x_{n+2})$ over $\Q$, extending $\sigma$ to $L_{2n+2}$ as before.

Now we have $\sigma$ defined on $\bigcup_{n=1}^{\infty}L_n$.  If $\bigcup_{n=1}^{\infty}L_n \ne \overline{\Q}$, then extend $\sigma$ to $\overline{\Q}$.  

In the even stages, we have arranged that $\sigma$ fixes an infinite extension of $\Q$, while in the odd stages we have arranged that $\sigma$ does not fix any infinite degree c.e.\ field, so $\sigma$ is random.  Moreover, at each stage $s$ of the construction, we had infinitely many options for the choice of $x_s$ --- indeed, there are infinitely many options for $x_s$ each of which leads to a different field $L_s$ --- so the total number of random automorphisms that can be constructed in this way is $2^{\aleph_0}$.
\end{proof}

\begin{remark}
Naturally, if $\sigma$ is random, then $\sigma^{-1}$ must be random as well since they both have the same fixed field.  However, $\sigma \circ \sigma^{-1}$ is clearly not random, so that the set of random automorphisms is not closed under composition.
\end{remark}
\begin{corollary}
There exists a random field.
\end{corollary}

\begin{proof}
Let $\sigma$ be a random element as constructed in the proof of Proposition \ref{ExRandomElts}.  Observe that, by construction, the fixed field of $\sigma$ is of infinite degree over $\Q$.  Then let $G = \langle \sigma \rangle$ and $F$ be the fixed field of $G$.  Now the absolute Galois group of $F$ will contain $G$, and so it will contain a random element.
\end{proof}

The following result is immediate from examining the technique of proof of Proposition \ref{ExRandomElts}.

\begin{corollary}
There are $2^{\aleph_0}$ distinct random fields in $\overline{\Q}$.
\end{corollary}

\begin{corollary} Let $F_1, \dots, F_k$ be random fields.  Then
\begin{enumerate}
\item If $J$ is an infinite extension of $\Q$ and $J \subseteq F_1$, then $J$ is random.
    \item If $F=\bigcap\limits_{i=1}^k F_i$ is an infinite extension of $\Q$, then $F$ is random.
\end{enumerate}
\end{corollary}

\begin{proof} Item 1 is immediate from the definition of a random field, since the absolute Galois group of $J$ contains that of $F_1$.  Then Item 2 follows from Item 1.
\end{proof}

\subsubsection{Does there exist a ``super random" group?}
One natural question one could ask is whether there exists a ``super random" absolute Galois group, that is the group where every non-trivial element is random.  The fixed field of such a group would have no extension containing an infinite degree c.e.\ subfield.  Existence of such a ``super random" field corresponds to the existence of a subset $X \subset \N$ such that for any co-infinite superset $S\supset X$ we have that $S$ contains no infinite c.e.\ subset.  Unfortunately, as is shown by the argument in Lemma \ref{star} such a set $X$ does not exist.  We summarize this finding in the following proposition.

\begin{proposition}
Every random absolute Galois group contains a non-random non-trivial element. Equivalently, every random field is contained in a non-random field not equal to the algebraic closure.
\end{proposition}

\section{The Set of Indices for Random Fields}\label{SecIdx}

The usual representation of a field in effective structure theory regards the field as having universe $\mathbb{N}$, and then identifies the field with its atomic diagram; that is, with the set of polynomial equations and inequations of elements true in that field.  In particular, if the field is computable, or even computably enumerable, it may be identified by the index of a Turing machine whose range is that atomic diagram.

A common way to calibrate the complexity of a class of structures is to determine the Turing degree of the set of indices of its members \cite{idxsets,ecl1,ecl2,idxhsr}.  Suppose now that we fix an oracle of Turing degree $\mathbf{d}$.  In this section, we attempt to calculate the Turing degree of the set of $\mathbf{d}$-indices for $\mathbf{d}$-computable random fields.  Of course, if $\mathbf{d} = \mathbf{0}$, there are none.  On the other hand, for other $\mathbf{d}$, the problem becomes nontrivial.


\begin{proposition}
Let $X \subseteq \mathbb{N}$.  The set of indices for $X$-computably enumerable subfields of $\overline{\Q}$ with infinite degree is $m$-complete $\Pi^0_2(X)$.  Moreover, the same result holds when the language is expanded to include constant symbols for generators for the field over $\Q$.
\end{proposition}

\begin{proof} We give a uniform proof, but omit the notation of $X$.  To show that the set of indices for infinite degree fields is $\Pi^0_2$, we will write a computable infinitary formula (see Section \ref{SecComputability}) which is true exactly in infinite degree fields.  By a standard result of Ash (see \cite{ashknight}), this will show that the set of indices must be $\Pi^0_2$.  

We first let $I_n$ denote the (decidable \cite{friedjarden}) set of irreducible polynomials over $\Q$ of degree strictly greater than $n$, and notice that the following collection of subfields $F \subset \overline{\Q}$ is exactly the set of infinite algebraic extensions of $\rat$.
\[ \left\{F |\bigwedge\limits_{n \in \mathbb{N}}\hspace{-0.2in}\bigwedge \left(\exists \alpha\in F  \bigvee\limits_{p \in I_n}\hspace{-0.22in}\bigvee p(\alpha)=0\right)\right\}.
\]

Toward completeness, we note that the set of indices for infinite computably enumerable sets is $m$-complete $\Pi^0_2$ (see, for instance, Theorem 4.3.2 of \cite{soare}).  To show that the set of indices for infinite extensions of $\rat$ is $m$-complete $\Pi^0_2$ it suffices to give a computable function $f$ such that $f(e)$ is the index for an infinite extension of $\rat$ if and only if $e$ is the index of an infinite c.e.\ set.

It suffices, then, to give a uniformly computable sequence of fields $(F_e : e \in \mathbb{N})$ such that $F_e$ is an infinite extension of $\rat$ if and only if $W_e$ is infinite.  We let $(K_i : i \in \mathbb{N})$ be a uniformly computable sequence of fields (given by their atomic diagrams) with \[\Q = K_0 \subsetneq K_1\subsetneq  \cdots \subsetneq \overline{\Q},\] where $\overline{\Q} = \bigcup\limits_{i\in \mathbb{N}} K_i$, and $K_{i+1}$ is a finite extension of $K_i$.  We note that for any $i$, the field $K_i$ contains all $K_j$ for $j<i$, so that any union, finite or infinite, of these fields will still be a field.  

Now, to each computably enumerable set $W_e$ we associate the subfield $F_e = \bigcup\limits_{i \in W_e} K_i$.  The result follows, since every infinite c.e.\ set $W_e$ will give $F_e = \overline{\Q}$, and every finite c.e.\ set $W_e$ will give a finite extension $F_e \supseteq \Q$.  Effectively recovering the indices of Turing machines for the fields $F_e$, we have the desired function $f$ on Turing machine indices. 
\end{proof}

\begin{corollary}
A set $X''$ can generate a listing of all $X$-computably enumerable absolute Galois groups of measure zero via their graphs.
\end{corollary}

\begin{proof} By Proposition \ref{FieldsAndGaloisTeq}, fields are uniformly Turing equivalent to the graphs of their absolute Galois groups.  The Corollary, then, is equivalent to enumerating the $X$-computably enumerable infinite extension fields of $\rat$.  By the previous proposition, this set is $\Pi^0_2(X)$.  Such a set is certainly enumerable from $X''$.
\end{proof}

\begin{proposition}\label{CheckIsec}
Let $G_c$ be a computably enumerable absolute Galois group.  Let $G$ be an arbitrary absolute Galois group. Then $\Gamma(G)''$ can determine whether $G_c \cap G$ is trivial.  
\end{proposition}

\begin{proof} By Corollary \ref{intersectinggroups}, it suffices to check whether, for each $i$ there is some $j>i$ for which there exists a nontrivial element of $\Gamma(G_c)(\alpha_j) \cap \Gamma(G)(\alpha_j)$.  This can be done in the second jump of $\Gamma(G)$, since $G_c$ is computably enumerable.
\end{proof}

\begin{corollary}
Let $G$ be as above.  Then $\Gamma(G)'''$ can determine whether $G$ is random.
\end{corollary}

\begin{proof}
By Lemma \ref{EnumeratingGroups}, we can enumerate the computably enumerable subgroups of $\Gal(\overline{\Q}/\Q)$ under $\emptyset'$.  Then for each of these groups we will perform the computation of Proposition \ref{CheckIsec}.  The final judgment on the randomness of $G$ is determined by whether, for every group enumerated, we have an empty intersection, each of which can be done by $\Gamma(G)''$, so that the full computation can be completed under $\Gamma(G)'''$, as required.\end{proof}


\bibliographystyle{plain}%
\bibliography{RandomFields}
\end{document}